\documentclass[a4paper]{article}
\usepackage{hyperref}
\usepackage{graphicx}
\usepackage{mathrsfs}
\usepackage{enumerate} 
\usepackage{amsxtra,amssymb,latexsym, amscd,amsthm}
\usepackage{indentfirst}
\usepackage{color}
\usepackage[utf8]{inputenc}
\usepackage[mathscr]{eucal}
\usepackage{amsfonts}
\usepackage{graphics}
\usepackage{multirow}
\usepackage{array}
\usepackage{subfigure}
\usepackage{cite}
\usepackage{wrapfig}
\usepackage{tikz}
\usepackage{diagbox}
\usepackage{tikz-cd}

\usepackage{pgfplots}
\pgfplotsset{compat=1.15}
\usepackage{mathrsfs}
\usetikzlibrary{arrows}

\newcommand{\footremember}[2]{%
    \footnote{#2}
    \newcounter{#1}
    \setcounter{#1}{\value{footnote}}%
}
 
\def\R{{\mathbb R}}
\def\Q{{\mathbb Q}}

\def\N{{\mathbb N}}

\DeclareMathOperator{\rank}{rank}
\DeclareMathOperator{\FJ}{FJ}
\DeclareMathOperator{\KKT}{KKT}
\DeclareMathOperator{\diag}{diag}

\DeclareMathOperator{\LT}{LT}

\newtheorem{theorem}{\bf Theorem}%[section]
\newtheorem{lemma}{\bf Lemma}
\newtheorem{algorithm}{\bf Algorithm}%[section]
\newtheorem{example}{\bf Example}%[section]
\newtheorem{remark}{\bf Remark}%[section]
\providecommand{\keywords}[1]
{
  \small	
  \textbf{\textbf{Keywords:}} #1
}
\begin{document}
\definecolor{qqzzff}{rgb}{0,0.6,1}
\definecolor{ududff}{rgb}{0.30196078431372547,0.30196078431372547,1}
\definecolor{xdxdff}{rgb}{0.49019607843137253,0.49019607843137253,1}
\definecolor{ffzzqq}{rgb}{1,0.6,0}
\definecolor{qqzzqq}{rgb}{0,0.6,0}
\definecolor{ffqqqq}{rgb}{1,0,0}
\definecolor{uuuuuu}{rgb}{0.26666666666666666,0.26666666666666666,0.26666666666666666}
\newcommand{\vi}[1]{\textcolor{blue}{#1}}
\newif\ifcomment
\commentfalse
\commenttrue
\newcommand{\comment}[3]{%
\ifcomment%
	{\color{#1}\bfseries\sffamily#3%
	}%
	\marginpar{\textcolor{#1}{\hspace{3em}\bfseries\sffamily #2}}%
	\else%
	\fi%
}

\newcommand{\mapr}[1]{{{\color{blue}#1}}}
\newcommand{\revise}[1]{{{\color{blue}#1}}}

\title{A symbolic algorithm for exact polynomial optimization strengthened with Fritz
John conditions}

\author{%
Ngoc Hoang Anh Mai\footremember{1}{CNRS; LAAS; 7 avenue du Colonel Roche, F-31400 Toulouse; France.}
  }
%\date{March 29, 2019}

\maketitle

\begin{abstract}
Consider a polynomial optimization problem.
Adding polynomial equations generated by the Fritz John conditions to the constraint set does not change the optimal value.
As proved in [arXiv:2205.04254 (2022)], the objective polynomial has finitely many values on the new constraint set under some genericity assumption.
Based on this, we provide an algorithm that allows us to compute exactly this optimal value.
Our method depends on the computations of real radical generators and Gr\"obner basis.
Finally, we apply our method to solve some instances of mathematical program with complementarity constraints.
\end{abstract}
\tableofcontents

\section{Introduction}
Besides numerical methods, symbolic tools that allow us to obtain the exact optimal value for a polynomial optimization problem have various applications in engineering science. 
Therefore developing such methods to address the exactness challenge becomes essential.

\paragraph{Polynomial optimization with singularities.}
Under the smoothness assumption on the constraint sets,  we can use the method of Greuet, Guo, Safey El Din, and Zhi in \cite{greuet2012global} to compute the exact optimal solution for a polynomial optimization problem.
Aside from this, the work of Greuet and Safey El Din \cite{greuet2014probabilistic} allows us to handle the case where the constraint sets have finitely many singularities.

In our previous works \cite{mai2022exact,mai2022explicit}, we give positive answers to two open questions (stated by Nie in \cite[Section 6]{nie2013exact}) concerning the existence and complexity of non-negativity certificates on a semi-algebraic set possibly having singularities.
Despite this theoretical guarantee of exactness, Lasserre's hierarchies of semidefnite relaxations \cite{lasserre2001global} based on these certificates have not yielded satisfactory results in such singular cases.
Hence we rely on the results in \cite{mai2022exact,mai2022explicit} to provide a symbolic algorithm in this paper for computing the exact optimal values of polynomial optimization problems whose optimal solutions are possibly singularities.

\paragraph{Problem statement.}
Let $\R[x]$ (resp. $\Q[x]$) stand for the ring of polynomials with real (resp. rational) coefficients in the vector of variables $x$.
We denote by $\R_r[x]$  the linear space of polynomials in $\R[x]$ over $\R$ of degree at most $r$.
Given $f,g_1,\dots,g_m\in\R[x]$, consider polynomial optimization problem:
\begin{equation}\label{eq:pop}
    f^\star:=\inf\limits_{x\in S(g)} f(x)\,,
\end{equation}
where $S(g)$ is the basic semi-algebraic set associated with $g=(g_1,\dots,g_m)$, i.e.,
\begin{equation}
    S(g):=\{x\in\R^n\,:\,g_j(x)\ge 0\,,\,j=1,\dots,m\}\,.
\end{equation}
\paragraph{First-order optimality conditions.}
The gradient of $p\in\R[x]$, denoted by $\nabla p$, is defined by $\nabla p=(\frac{\partial p}{\partial x_1},\dots,\frac{\partial p}{\partial x_n})$.
We say that the Fritz John conditions hold for problem \eqref{eq:pop} at $u\in S(g)$  if 
\begin{equation}\label{eq:FJcond}
\begin{cases}
    		\exists (\lambda_0,\dots,\lambda_m)\in [0,\infty)^{m+1}\,:\\
          \lambda_0 \nabla f(u)=\sum_{j=1}^m \lambda_j \nabla g_j(u)\,,\\
          \lambda_j g_j(u) =0\,,\,j=1,\dots,m\,,\\
          \sum_{j=0}^m \lambda_j^2=1
    \end{cases}
   \Leftrightarrow
    \begin{cases}
         \exists (\lambda_0,\dots,\lambda_m)\in\R^{m+1}\,:\\
         \lambda_0^2 \nabla f(u)=\sum_{j=1}^m \lambda_j^2 \nabla g_j(u)\,,\\
          \lambda_j^2 g_j(u) =0\,,\,j=1,\dots,m\,,\\
          \sum_{j=0}^m \lambda_j^2=1\,.
    \end{cases}
\end{equation}
In addition, the Karush--Kuhn--Tucker conditions hold for problem \eqref{eq:pop} at $u\in S(g)$  if 
\begin{equation}\label{eq:KKT.cond}
\begin{cases}
    		\exists (\lambda_1,\dots,\lambda_m)\in[0,\infty)^{m}\,:\\
          \nabla f(u)=\sum_{j=1}^m \lambda_j \nabla g_j(u)\,,\\
          \lambda_j g_j(u) =0\,,\,j=1,\dots,m
    \end{cases}
       \Leftrightarrow
    \begin{cases}
         \exists (\lambda_1,\dots,\lambda_m)\in\R^{m+1}\,:\\
         \nabla f(u)=\sum_{j=1}^m \lambda_j^2 \nabla g_j(u)\,,\\
          \lambda_j^2 g_j(u) =0\,,\,j=1,\dots,m\,.
    \end{cases}
\end{equation}
If $u$ is a local minimizer for problem \eqref{eq:pop}, then the Fritz John conditions hold for problem \eqref{eq:pop} at $u$.
The Karush--Kuhn--Tucker conditions can be considered a particular case of the Fritz John conditions.
When some constraint qualification holds at $u$, the Karush--Kuhn--Tucker conditions hold for problem \eqref{eq:pop} at $u$.
However, there exist cases where the Karush--Kuhn--Tucker conditions do not hold for problem \eqref{eq:pop} at any local minimizer of this problem.
Interested readers are referred to \cite{pham2019tangencies,pham2019optimality,
guo2020types,ha2009solving,pham2019optimality} theoretical and practical results involving first-order optimality conditions in polynomial optimization.

In the unconstrained case of problem \eqref{eq:pop} (i.e., $S(g)=\R^n$), the Fritz John and Karush--Kuhn--Tucker conditions (\eqref{eq:FJcond} and \eqref{eq:KKT.cond}) reduce to $\nabla f(u)=0$. 
In this case, Magron, Safey El Din, and Vu \cite{gradsos} rely on the sum-of-squares strengthenings with gradient ideals by Nie, Demmel, and Sturmfels \cite{nie2006minimizing} to provide an exact method for computing $f^\star$ when the coefficients of $f$ are rational.

\paragraph{Motivation for the Fritz John conditions.}
In nonlinear programming, constraint qualifications are known to be sufficient conditions for the Karush--Kuhn--Tucker conditions.
However, in many recent applications, constraint qualifications have not  been satisfied, which leads to limitations of the classical methods using the Karush--Kuhn--Tucker conditions.
Mathematical program with complementarity constraints \cite{albrecht2017mathematical} is a class of challenging optimization problems that belong to such a situation.
It is because complementarity constraints typically result in the violation of all standard constraint qualifications.
Our method based on the Fritz John conditions has the potential to overcome this limitation as it does not requires any constraint qualification.
We illustrate this on some instances of mathematical programs with complementarity constraints (see Example \ref{exam:MPCC}).

\paragraph{Polynomials from first-order optimality conditions.}
Given $h_1,\dots,h_l\in\R[x]$, let $V(h)$ be the (real) variety generated by $h=(h_1,\dots,h_l)$ defined  by
\begin{equation}
V(h):=\{x\in\R^n\,:\,h_j(x)=0\,,\,j=1,\dots,l\}\,.
\end{equation}
Let $\bar\lambda=(\lambda_0,\lambda_1,\dots,\lambda_m)$ be a vector of $m$ variables.
Set $\lambda=(\lambda_1,\dots,\lambda_m)$.
Given $f\in\R[x]$, $g=(g_1,\dots,g_m)$ with $g_j\in\R[x]$, we define:
\begin{equation}\label{eq:.polyFJ}
    h_{\FJ}:=(\lambda_0\nabla f-\sum_{j=1}^m \lambda_j \nabla g_j,\lambda_1g_1,\dots,\lambda_mg_m,1-\sum_{j=0}^m\lambda_j^2)\,,
\end{equation}
\begin{equation}\label{eq:.polyFJ.plus}
    h_{\FJ}^+:=(\lambda_0^2\nabla f-\sum_{j=1}^m \lambda_j^2 \nabla g_j,\lambda_1^2g_1,\dots,\lambda_m^2g_m,1-\sum_{j=0}^m\lambda_j^2)\,,
\end{equation}
\begin{equation}\label{eq:.polyKKT}
    h_{\KKT}:=(\nabla f-\sum_{j=1}^m \lambda_j \nabla g_j,\lambda_1g_1,\dots,\lambda_mg_m)\,,
\end{equation}
\begin{equation}\label{eq:.polyKKT.plus}
    h_{\KKT}^+:=(\nabla f-\sum_{j=1}^m \lambda_j^2 \nabla g_j,\lambda_1^2g_1,\dots,\lambda_m^2g_m)\,.
\end{equation}
Here $h_{\FJ}$, $h_{\FJ}^+\subset \R[x,\bar \lambda]$  include polynomials from the Fritz John conditions stated in \eqref{eq:FJcond} and $h_{\KKT},h_{\KKT}^+\subset \R[x, \lambda]$ include polynomials from the Karush--Kuhn--Tucker conditions stated in \eqref{eq:KKT.cond}.
Set
\begin{equation}
\begin{array}{rl}
&W_1:=(S(g)\times \R^{m+1}) \cap V(h_{\FJ})\,,\, \quad W_2:=(S(g)\times \R^{m+1}) \cap V(h_{\FJ}^+),\\
& W_3:=(S(g)\times \R^{m}) \cap V(h_{\KKT})\,,\, \quad W_4=(S(g)\times \R^{m}) \cap V(h_{\KKT}^+)\,.
\end{array}
\end{equation}
If the Fritz John (resp. Karush--Kuhn--Tucker) conditions hold for problem \eqref{eq:pop} at $u$, then $u$ is in the projections of $W_1$ and $W_2$ (resp. $W_3$ and $W_4$) onto the linear space spanned by the first $n$ coordinates and 
\begin{equation}\label{eq:opt.val.cond}
f^\star=\min f(W_1)=\min f(W_2)\quad \text{(resp. $f^\star=\min f(W_3)=\min f(W_4)$)}\,.
\end{equation}
Here $f^\star$ is defined as in \eqref{eq:pop}.

\paragraph{Contribution.}
Our goal in this paper is to design a symbolic algorithm to compute the exact optimal value $f^\star$ attained by $f$ over $S(g)$ (problem \eqref{eq:pop}) under mild conditions.
Let $s\in\{1,2,3,4\}$ be fixed.
The idea of our method is as follows:

Relying on the computations of Gr\"obner basis and real radical generators, we obtain univariate polynomials $\hat h_1,\dots,\hat h_r$ whose set of common real zeros is identically the Zariski closure of the image $f(W_s)$.
Let $d$ be the upper bound on the degrees of $f,g_i,h_j$.
Then each $\hat h_t$ has degree not larger than
\begin{equation}\label{eq:bound.degree.univar.poly}
2\times d^{2^{{O((n+2m)^2)}^{2^{n+2m}}}}\,.
\end{equation}

Moreover, we prove that $f(W_s)$ coincides with its Zariski closure when the image $f(W_s)$ has finitely many elements.
In this case, we get all elements of $f(W_s)$ by solving the system of the univariate polynomials $\hat h_1=\dots=\hat h_r=0$.
By \eqref{eq:opt.val.cond}, $f^\star$ is the smallest one among these elements.

It remains to provide sufficient conditions for $f(W_s)$ to be finite. 
Thanks to our previous works in \cite{mai2022exact,mai2022explicit}, the image $f(W_1)$ (resp. $f(W_2)$) is finite when the set of critical points $C(g)$ (resp. $C^+(g)$) defined later in  \eqref{eq:crit.point} (resp. \eqref{eq:crit.point.plus}) is finite.
Moreover, if the Karush--Kuhn--Tucker conditions hold for problem \eqref{eq:pop} at some global minimizer, the images $f(W_3)$ and $f(W_4)$ are finite.

Due to its high computational cost \eqref{eq:bound.degree.univar.poly},  we illustrate our method with some modest-sized examples.
\paragraph{Related work.}
Greuet and Safey El Din provide in \cite{greuet2014probabilistic} a probabilistic algorithm for solving POP on a real variety $V(h)$, where $h=(h_1,\dots,h_l)$ with $h_j\in\R[x]$. 
They can extract a solution under the following regularity assumptions: the ideal generated by $h$ is radical, $V(h)$ is equidimensional of positive dimension, and $V(h)$ has finitely many singular points. 
We emphasize that our method for problem \eqref{eq:pop} does not require the finiteness of the set of critical points $C(g)$ (resp. $C^+(g)$) but only requires the finiteness of the image of $C(g)$ (resp. $C^+(g)$) under $f$. % (see Section \ref{sec:examples}).

\paragraph{Organization.}
We organize the paper as follows: 
Section \ref{sec:prelimi} is to recall some necessary tools from real algebraic geometry. 
Sections \ref{sec:image.variety} and \ref{sec:image.semial} are to construct algorithms for computing the Zariski closures of the images of a real variety and a semi-algebraic set under a polynomial together with the dimensions of these images.
Section \ref{sec:application} presents our main algorithm for computing the exact optimal value for a polynomial optimization problem under mild conditions.
We also illustrate our method with some interesting examples in this section.
\section{Preliminaries}
\label{sec:prelimi}
This section presents some preliminaries from real algebraic geometry needed to construct our main algorithms.
\subsection{Ideals}
Denote by $\Sigma^2[x]$ (resp. $\Sigma^2_r[x]$) the cone of sum of squares of polynomials in $\R[x]$ (resp. $\R_r[x]$).
Given $h_1,\dots,h_l\in\R[x]$, we call $I(h)$ the ideal generated by $h=(h_1,\dots,h_l)$, if
\begin{equation}
    I(h):= \sum_{j=1}^l h_j \R[x]\,.
\end{equation}
Note that the intersection of two ideals is also an ideal.
The real radical of an ideal $I(h)$, denoted by $\sqrt[\R]{I(h)}$, is defined as
\begin{equation}
{\sqrt[\R]{I(h)}}=\{f\in\R[x]\,:\,\exists m\in \N\,:\,-f^{2m}\in\Sigma^2[x]+I(h)\}\,.
\end{equation}
Krivine--Stengle's Nichtnegativstellens\"atze \cite{krivine1964anneaux} imply that
\begin{equation}\label{eq:real.radi}
\sqrt[\R]{I(h)}=\{p\in\R[x]\,:\,p=0\text{ on }V(h)\}\,.
\end{equation}
We say that $I(h)$ is real radical if $I(h)=\sqrt[\R]{I(h)}$.
\begin{remark}\label{rem:radical.gener}
Becker and Neuhaus provide a method in \cite{becker1993computation,neuhaus1998computation} to compute a generator of the real radical $\sqrt[\R]{I}$ with a given generator of an ideal $I\subset\R[x]$. 
The ideal is as follows: They first reduce the general situation to the univariate case by using factorization and quantifier elimination.
They then handle zero-dimensional ideals after localizing and contracting the one-dimensional ideals.
Thereby the degrees of elements in the generator of $\sqrt[\R]{I}$ are bounded from above by $d^{2^{
 O(n^2)}}$ if polynomials in generator of $I$ have degrees at most $d$.
 
Safey El Din, Yang, and Zhi provide in \cite{el2021computing} a probabilistic algorithm with complexity $(rnd^n)^{O(1)}$ to compute a generator of $\sqrt[\R]{I}$ from a generator $(h_1,\dots,h_r)\subset \Q[x]$ of $I$.
Another method of Baldi and Mourrain \cite{baldi2021computing} to compute a generator of $\sqrt[\R]{I}$ depends on the numerical solutions of moment relaxations.
\end{remark}
Given $A\subset \R^n$, let $I(A)$ denote the vanishing ideal of $A$, i.e., 
\begin{equation}
I(A)=\{f\in\R[x]\,:\,f=0\text{ on }A\}\,.
\end{equation}
Given $J$ being a finite subset of $\R[x]$, let $V(J)$ stand for the real variety defined by $J$, i.e., 
\begin{equation}
V(J)=\{x\in\R^n\,:\,p(x)=0\,,\,\forall p\in J\}\,.
\end{equation}
The Zariski closure of a subset $A\subset\R^n$, denoted by $Z(A)$, is the smallest real variety containing $A$. 
By \cite[Proposition 1, Section 4, Chapter 4]{cox2013ideals}, it holds that
\begin{equation}\label{eq:zariski.def}
Z(A)=V(I(A))\,.
\end{equation}

\subsection{Gr\"obner basis}
Given $p=\sum_{j=1}^rp_{\alpha_j} x^{\alpha_j}\in\R[x]$ such that $p_{\alpha_j}\ne 0$ and $x^{\alpha_r}>\dots>x^{\alpha_r}$, according to lex order, we call $p_{\alpha_r} x^{\alpha_r}$ the leading term of $p$, denote by $\LT(p)$.
Given $I\subset\R[x]$ being an ideal other than $\{0\}$, we denote by $\LT(I)$ the set of leading terms of elements of $I$.

A finite subset $G=\{h_1,\dots,h_r\}\subset\R[x]$ generating an ideal $I$ is called a Gr\"obner basis if
$V(\LT(h_1),\dots,\LT(h_r))=V(\LT(I))$. 
\begin{remark}\label{rem:bound.Grobner}
Buchberger suggests in \cite{buchberger1976theoretical} an algorithm to compute a Gr\"obner basis $h^G$ of an ideal $I$ with a given generator $h$ (see also F4 and F5 algorithms by Faug\`ere in \cite{faugere1999new,faugere2002new}).
For every couple $h_i, h_j$ in $h$, denote by $a_{ij}$ the least common multiple of $\LT(h_i)$ and $\LT(h_j)$.
Buchberger's algorithm consists of the following steps:
\begin{enumerate}
\item Set $h^G := h$.
\item Choose two polynomials $h_i,h_j$ in $h^G$ and let 
$s_{ij} = \frac{a_{ij}}{LT(h_i)}h_i-\frac{a_{ij}}{LT(h_j)}h_j$.
\item Reduce $s_{ij}$, with the multivariate division algorithm relative to the set $h^G$ until the result is not further reducible. If the result is non-zero, add it to $h^G$.
\item Repeat Steps 2 and 3.
\end{enumerate}
Assume that each element of the generator $h$ has degree at most $d$.
 Dub\'e proves in \cite{dube1990structure} that Buchberger's algorithm has complexity
\begin{equation}\label{eq:complex.radical}
d^{2^{n+o(1)}}
\end{equation} 
to produce the Gr\"obner basis $h^G$, whose elements have degrees at most $2\left({\frac {d^{2}}{2}}+d\right)^{2^{n-2}}$.

To explore recent developments of Groebner basis, we refer the readers to \cite{berthomieu2022new,berthomieu2022guessing,
berthomieu2015linear,berthomieu2022faster,
faugere2016complexity}.
\end{remark}

We restate the elimination theorem (see, e.g., \cite[Theorem 2, Section 1, Chapter 3]{cox2013ideals}) in the following lemma:

\begin{lemma} Let $I \subset \R[x]$ be an ideal and let
$h^G$ be a Gr\"obner basis of $I$ w.r.t. lex order where $x_1 \succ x_2 \succ \dots \succ x_n$. Then
for $j=1,\dots,n$, the set
$h^G \cap \R[x_{j+1},\dots,x_n]$
is a Gr\"obner basis of the ideal $I\cap \R[x_{j+1},\dots,x_n]$.
\end{lemma}

\subsection{Semi-algebraic sets}
A semi-algebraic subset of $\R^n$ is a subset of the form
\begin{equation}\label{eq:def.semi.set}
\bigcup_{i=1}^t\bigcap_{j=1}^{r_i}\{x\in\R^n\,:\,f_{ij}(x)*_{ij}0\}\,,
\end{equation}
where $f_{ij}\in\R[x]$ and $*_{ij}$ is either $>$ or $=$.
Equivalently, a semi-algebraic set is a finite union of basic semi-algebraic sets.
We define the dimension of a semi-algebraic subset $S$ of $\R^n$, denoted by $\dim(S)$, to be the highest dimension at points at which $S$ is a real submanifold.

We recall some properties of semi-algebraic sets in the following lemma:
\begin{lemma}\label{lem:properties.semi-algebraic.set}
The following statements hold:
\begin{enumerate}
\item The image of a semi-algebraic set under a polynomial is semi-algebraic.
\item A semi-algebraic subset of $\R$ is a finite union of intervals and points in $\R$.
\end{enumerate}
\end{lemma}
The first statement of Lemma \ref{lem:properties.semi-algebraic.set} is proved by Korkina and Kushnirenko  in \cite{korkina1985another}.
The second statement is given in \cite[Example 1.1.1]{pham2016genericity}.

Let $x_{n+1}$ be a single-variable independent from $x$.
We provide the following two lemmas which characterize the image of a semi-algebraic set under a polynomial:
\begin{lemma}\label{lem:project}
Let $p\in\R[x_{n+1}]$, $f\in\R[x]$ and $h=(h_1,\dots,h_l)$ with $h_j\in\R[x]$.
We define the projection
\begin{equation}\label{eq:proj}
\hat \pi:\R^{n+1}\to\R\,,\quad (x,x_{n+1})\mapsto \hat \pi(x,x_{n+1})=x_{n+1}\,.
\end{equation}
Then it holds that $\hat \pi(V(h,x_{n+1}-f))=f(V(h))$.
\end{lemma}
\begin{proof}
The result follows from the following equivalences:
\begin{equation}
\begin{array}{rl}
&a_{n+1}\in f(V(h))\\
\Leftrightarrow &\exists a\in V(h)\,:\, a_{n+1}=f(a)\\
\Leftrightarrow &\exists a\in \R^{n}\,:\, (a,a_{n+1})\in V(h,x_{n+1}-f)\\
\Leftrightarrow &a_{n+1}\in \hat \pi(V(h,x_{n+1}-f))\,.\\
\end{array}
\end{equation}
\end{proof}
\begin{lemma}\label{lem:dimequal}
Let $A$ be a closed semi-algebraic subset of $\R^n$ and $f$ be a polynomial in $\R[x]$. Then $\dim(f(A))=\dim(Z(f(A)))$.
Moreover, if $f(A)$ is finite, then $f(A)=Z(f(A))$.
\end{lemma}
\begin{proof} 
By Lemma \ref{lem:properties.semi-algebraic.set}, $f(A)$  is a finite union of intervals and points in $\R$.
Then one of the following two cases occurs:
\begin{itemize}
\item Case 1: $f(V(h))$ has non-empty interior. Then $f(V(h))$ has dimension one. We also get $Z(f(A))=\R$, which has dimension one. It implies that $\dim(f(A))=\dim(Z(f(A)))=1$.
\item Case 2: $f(V(h))$ is finite. Then it is easily seen that $f(A)=Z(f(A))$, which gives $\dim(f(A))=\dim(Z(f(A)))=0$.
\end{itemize}
Hence the result follows.
\end{proof}

\begin{remark}
The conclusion of Lemma \ref{lem:dimequal} still holds if we assume that $f$ is a semi-algebraic function (see \cite[Section 2.2.1]{coste2000introduction}).
\end{remark}

\subsection{Sets of critical points}
Given $g_1,\dots,g_m\in\R[x]$, let  $\varphi^g:\R^{n}\to \R^{(n+m)\times m}$ be a function associated with $g=(g_1,\dots,g_m)$ defined by
\begin{equation}
\varphi^g(x)=\begin{bmatrix}
\nabla g(x)\\
\diag(g(x))
\end{bmatrix}=
    \begin{bmatrix}
    \nabla g_1(x)& \dots& \nabla g_m(x)\\
    g_1(x)&\dots&0\\
    .&\dots&.\\
    0&\dots&g_m(x)
    \end{bmatrix}\,.
\end{equation}
Given a real matrix $A$, we denote by $\rank(A)$ the dimension of the vector space generated by the columns of $A$ over $\R$.
We say that a set $\Omega$ is finite if its cardinality is a non-negative integer.
Let  $C(g)$ be the set of critical points associated with $g$ defined by
\begin{equation}\label{eq:crit.point}
    C(g):=\{x\in\R^n\,:\,\rank(\varphi^g(x))< m\}.
\end{equation}
Given a real matrix $A$, we denote by $\rank^+(A)$ the largest number of columns of $A$ whose convex hull over $\R$ has no zero.
Let  $C^+(g)$ be the set of critical points associated with $g$ defined by
\begin{equation}\label{eq:crit.point.plus}
    C^+(g):=\{x\in \R^n\,:\,\rank^+(\varphi^g(x))< m\}.
\end{equation}
\begin{remark}
Let $\Phi^g:\R^{n+m}\to\R$ be a mapping defined by 
\begin{equation}
\Phi^g(x,y)=(g_1(x)e^{y_1},\dots,g_m(x)e^{y_m})\,.
\end{equation}
Then $\Phi^g$ is differentiable and the Jacobian of $\Phi^g$ is of the form
\begin{equation}
\nabla \Phi^g(x,y)=\begin{bmatrix}
    e^{y_1}\nabla g_1(x)& \dots& e^{y_m}\nabla g_m(x)\\
    e^{y_1}g_1(x)&\dots&0\\
    .&\dots&.\\
    0&\dots&e^{y_m}g_m(x)
    \end{bmatrix}=\varphi^g(x) \times \begin{bmatrix}
    e^{y_1}& \dots& 0\\
    .&\dots&.\\
    0&\dots&e^{y_m}
    \end{bmatrix}\,.
\end{equation}
It implies that $\rank(\nabla \Phi^g(x,y))=\rank(\varphi^g(x))$ since $e^{y_j}\ne 0$.
Note that the standard set of critical points of $\Phi^g$ is defined by 
\begin{equation}
C_{\Phi^{g}}=\{(x,y)\in\R^{n+m}\,:\,\rank(\nabla \Phi^g(x,y))< m\}\,.
\end{equation}
From this, we get $C_{\Phi^{g}}=C(g)\times \R^m$.
Thus our set of critical points $C(g)$ is the projection of $C_{\Phi^{g}}$ on the linear space spanned by the first $n$ coordinates.
\end{remark}
The following lemma relies on the proofs of non-negativity certificates stated in our previous works \cite{mai2022exact,mai2022explicit} and Demmel--Nie--Powers' work \cite{demmel2007representations}:
\begin{lemma}\label{lem:image.finite}
Let $f\in\R[x]$, $g=(g_1,\dots,g_m)$ with $g_j\in\R[x]$.
Let $h_{\FJ}$, $h_{\FJ}^+$, $h_{\KKT}$, $h_{\KKT}^+$ be defined as in \eqref{eq:.polyFJ}, \eqref{eq:.polyFJ.plus}, \eqref{eq:.polyKKT}, \eqref{eq:.polyKKT.plus}, respectively.
Assume that problem \eqref{eq:pop} has an optimal solution $x^\star$.
Then the following statements hold:
\begin{enumerate}
\item If $f(S(g)\cap C(g))$ is finite, $f((S(g)\times \R^{m+1})\cap V(h_{\FJ}))$ is non-empty and finite.
\item If $f(S(g)\cap C^+(g))$ is finite, $f((S(g)\times \R^{m+1})\cap V(h_{\FJ}^+))$ is non-empty and finite.
\item If the Karush--Kuhn--Tucker conditions hold for problem \eqref{eq:pop} at $x^\star$, 
\begin{equation}
f((S(g)\times \R^{m})\cap V(h_{\KKT}))\quad\text{ and }\quad f((S(g)\times \R^{m})\cap V(h_{\KKT}^+))
\end{equation}
are non-empty and finite.
\end{enumerate}
\end{lemma} 
\begin{remark}
We call $f(S(g)\cap C(g))$ and $f(S(g)\cap C^+(g))$ in Lemma \ref{lem:image.finite} the sets of critical values of $f$ on $S(g)$.
Studies of similar types to these critical values are found in \cite{guo2010global,faugere2012grobner,
berthomieu2022grobner,berthomieu2021towards,
berthomieu2022computing,berthomieu2021computation}.
\end{remark}
Note that the if-clauses in the three statements of Lemma \ref{lem:image.finite} hold generically according to \cite[Theorem 2]{mai2022exact}.
To prove the first statement of Lemma \ref{lem:image.finite}, we decompose the intersection $S(g) \cap V(h_{\FJ})$ into finitely many connected components and claim that the polynomial $f$ is constant on each connected component. 
We prove this by considering connected components not contained in the hyperplane $\lambda_0=0$ and then applying the mean value theorem.
The remaining case relies on the assumption that $f(S(g)\cap C(g))$ is finite. 
Similar arguments apply to the other statements of Lemma \ref{lem:image.finite}.
\begin{remark}
In Lemma \ref{lem:image.finite}, by adding polynomial equality constraints $h_\text{2ndFJ}\supset h_{\FJ}$ (resp. $h_\text{2ndKKT}\supset h_{\KKT}$) from the Fritz John (resp. Karush--Kuhn--Tucker) second-order necessary conditions (see, e.g., \cite[Section 2.4.1]{kungurtsev2013second}), we can reduce the size of the image under $f$  that contains $f^\star$, i.e., $f((S(g)\times \R^{m+1})\cap V(h_{\FJ}))\supset f((S(g)\times \R^{m+1})\cap V(h_\text{2ndFJ}))$ (resp. $f((S(g)\times \R^{m})\cap V(h_{\KKT}))\supset f((S(g)\times \R^{m})\cap V(h_\text{2ndKKT}))$).
\end{remark}
\paragraph{Additional discussion for our previous works.}
In our previous works \cite{mai2022exact,mai2022explicit}, we use the finiteness of the images of $S(g) \cap V(h_{\FJ})$ and $S(g) \cap V(h_{\FJ}^+)$ under $f$ to construct representations of $f-f^\star$ without denominators involving quadratic modules and preorderings associated with these two intersections. 
Here $f^\star$ is defined as in \eqref{eq:pop}.
These non-negativity certificates allow us to tackle the following two cases: (i) polynomial $f-f^\star$ is non-negative with infinitely many zeros on basic semi-algebraic sets $S(g)$ and (ii) the Karush–Kuhn–Tucker conditions do not hold for problem \eqref{eq:pop} at any zero of $f-f^\star$ on $S(g)$.

\paragraph{Singularities of semi-algebraic sets.} We call $\bar x\in S(g)$ a singularity of $S(g)$ if the Fritz John conditions \eqref{eq:FJcond} hold for the problem \eqref{eq:pop} at $\bar x$  without a positive multiplier on the objective gradient (i.e., $\lambda_0=0$).
In this case, the Karush–Kuhn–Tucker conditions do not hold for the problem \eqref{eq:pop} at $\bar x$, which implies that $\bar x\in C^+(g)$ and $\bar x\in C(g)$.
Previously Guo, Wang, and Zhou \cite{guo2014minimizing} have extended Nie's method \cite{nie2013exact} to rational objective functions and a finite number of singularities.
While their method is inapplicable to \cite[Example 5]{mai2022exact} because of its infinite number of singularities, our method applies to this example.
It is because we only use the finiteness assumption of the image of the singularities of $S(g)$ under $f$ that still includes the case of the infinite number of singularities.

\paragraph{Computational limitation of our previous method.} Based on the representations in our previous work \cite{mai2022exact,mai2022explicit}, we utilize Lasserre's hierarchy of semidefinite relaxations \cite{lasserre2001global} to provide a sequence of values finitely converges to the optimal value of a polynomial optimization problem with singularities.
However, we only obtain approximate (non-exact) values with low accuracy when using a numerical method to solve the corresponding semidefinite relaxations.
As shown in Section \ref{sec:examples}, we can get the exact optimal values using our new method in this paper.
\section{Zariski closure of the image of a real variety under a polynomial}
\label{sec:image.variety}
%\cite{harris2019computing}
The following algorithm allows us to obtain univariate polynomials whose set of common real zeros is identically the Zariski closure of the image of a given real variety under a given polynomial:
\begin{algorithm}\label{alg:image}
Computing the Zariski closure of the image of a real variety under a polynomial:
\begin{itemize}
\item Input: $f\in\R[x]$ and $h=(h_1,\dots,h_l)$ with $h_j\in\R[x]$.
\item Output: $\hat h=(\hat h_1,\dots,\hat h_r)$ with $\hat h_t\in\R[x_{n+1}]$.
\end{itemize}
\begin{enumerate}
\item Compute a generator $\check h\subset \R[x,x_{n+1}]$ of the real radical $\sqrt[\R]{I(h,x_{n+1}-f)}$.  
\item Compute a Gr\"obner basis $h^G$ of $I(\check h)$ w.r.t. lex order where $x_1\succ\dots\succ x_n\succ x_{n+1}$.
\item Set $\hat h:=h^G\cap \R[x_{n+1}]$.
\end{enumerate}
\end{algorithm}
\begin{remark}\label{rem:bound.image.variety}
Let $d$ be the largest degree of $f,h_j$ in the input of Algorithm \ref{alg:image}.
In the first step of this algorithm, we  utilize Becker--Neuhaus's method in \cite{becker1993computation,neuhaus1998computation} to compute real radical generators (see Remark \ref{rem:radical.gener}). 
Then each element of $\check h$ has degree upper bounded by
$d^{2^{O(n^2)}}$.
The second step of Algorithm \ref{alg:image} requires the computation of a Gr\"obner basis, which can be handled by using Buchberger's algorithm (see Remark \ref{rem:bound.Grobner}).
Then polynomials in $h^G$ have degrees not larger than 
\begin{equation}
2\left({\frac {1}{2}}\times d^{2\times 2^{O(n^2)}}+d^{2^{O(n^2)}}\right)^{2^{n-1}}\,.
\end{equation}
Thus the degree of each $\hat h_t$ is bounded from above by $2\times d^{2^{{O(n^2)}^{2^{n-1}}}}$.
It is open to reducing this degree bound.
\end{remark}

We apply the elimination theorem to prove the following lemma:
\begin{lemma}\label{lem:algimag.well}
Let $f\in\R[x]$ and $h=(h_1,\dots,h_l)$ with $h_j\in\R[x]$.
Let $\hat h=(\hat h_1,\dots,\hat h_r)$ with $\hat h_t\in\R[x_{n+1}]$ be the output of Algorithm \ref{alg:image} with input $f,h$.
Then $V(\hat h)=Z(f(V(h)))$.
\end{lemma}
\begin{proof}
By the first step of Algorithm \ref{alg:image}, it holds that $I(\check h)=\sqrt[\R]{I(h,x_{n+1}-f)}$.
From this, we get 
\begin{equation}\label{eq:equalities.varieties}
V(\check h)=V(h,x_{n+1}-f)=V(h^G)\,.
\end{equation}
since  $h^G$ is a Gr\"obner basis of $I(\check h)$.
Moreover, the elimination theorem (see, e.g., \cite[Theorem 2, Section 1, Chapter 3]{cox2013ideals}) yields that $\hat h$ is a Gr\"obner basis of the ideal $I(\check h)\cap \R[x_{n+1}]$.
Let $\hat \pi$ be as in \eqref{eq:proj}.
The result follows from the following equivalences:
\begin{equation}\label{eq:equivalences}
    \begin{array}{rl}
&  t\in V(\hat h)\\
\Leftrightarrow & \forall p\in I(\check h)\cap  \R[x_{n+1}]\,:\,p(t)=0\\
\Leftrightarrow & \forall p\in \R[x_{n+1}]\,:\,p\in I(\check h) \Rightarrow p(t)=0\\
\Leftrightarrow & \forall p\in \R[x_{n+1}]\,:\,(p=0\text{ on } V(\check h)) \Rightarrow p(t)=0\\
\Leftrightarrow & \forall p\in \R[x_{n+1}]\,:\,(p=0\text{ on } V(h,x_{n+1}-f)) \Rightarrow p(t)=0\\
\Leftrightarrow & \forall p\in \R[x_{n+1}]\,:\,(p=0\text{ on } \hat \pi(V(h,x_{n+1}-f))) \Rightarrow p(t)=0\\
\Leftrightarrow & \forall p\in \R[x_{n+1}]\,:\,(p=0\text{ on } f(V(h))) \Rightarrow p(t)=0\\
    \Leftrightarrow & \forall p\in I(f(V(h))) \Rightarrow p(t)=0\\
   \Leftrightarrow & t\in V(I(f(V(h))))\\
   \Leftrightarrow & t\in Z(f(V(h)))\,.
    \end{array}
\end{equation}
The sixth equivalence relies on Lemma \ref{lem:project}.
The final equivalence is based on \eqref{eq:zariski.def}.
\end{proof}

\begin{remark}
If we replace Step 1 of Algorithm \ref{alg:image} with ``Set $\check h:=(h,x_{n+1}-f)$.'', then $V(\hat h)\subset Z(f(V(h)))$.
To prove this, we process similarly to the proof of Lemma \ref{lem:algimag.well} and replace the third equivalence in \eqref{eq:equivalences} with an implication.
\end{remark}
The following algorithm enables us to compute the dimension of the image of a given variety under a given polynomial:
\begin{algorithm}\label{alg:dim}
Computing the dimension of the image of a real variety under a polynomial:
\begin{itemize}
\item Input: $f\in\R[x]$ and $h=(h_1,\dots,h_l)$ with $h_j\in\R[x]$.
\item Output: $d\in\N$.
\end{itemize}
\begin{enumerate}
\item Compute $\hat h=(\hat h_1,\dots,\hat h_r)$ with $\hat h_t\in\R[x_{n+1}]$ such that $V(\hat h)=Z(f(V(h)))$ by using Algorithm \ref{alg:image}.  
\item Set $d:=\dim(V(\hat h)))$.
\end{enumerate}
\end{algorithm}
\begin{remark}
In the second step of Algorithm \ref{alg:dim}, if $V(\hat h)\ne \R$, then $V(\hat h)$ has finitely many reals, which implies $d=0$.
In particular, if $V(\hat h)=\emptyset$ (e.g., $\hat h=\{1\}$), then $d=0$.
Moreover, if $\hat h=\emptyset$ or $\hat h=\{0\}$, then $V(\hat h)=\R$ and $d=1$.
\end{remark}

The following lemma is a direct consequence of Lemmas \ref{lem:dimequal} and \ref{lem:algimag.well}:
\begin{lemma}\label{lem:alg:dim}
Let $f\in\R[x]$ and $h=(h_1,\dots,h_l)$ with $h_j\in\R[x]$.
Let $d$ be the output of Algorithm \ref{alg:dim} with input $f$ and $h$.
Then $d$ is the dimension of the image $f(V(h))$.
\end{lemma}
\begin{remark}
To compute the dimension of a general real algebraic variety, we refer the readers to the work of Lairez and Safey El Din in \cite{lairez2021computing}.
\end{remark}

\section{Zariski closure of the image of a basic semi-algebraic set under a polynomial}
\label{sec:image.semial}
The following algorithm produces a finite number of univariate polynomials whose set of common real zeros is identically the Zariski closure of the image of a given basic semi-algebraic set under a given polynomial:
\begin{algorithm}\label{alg:image.semi}
Computing the Zariski closure of the image of a basic semi-algebraic set under a polynomial:
\begin{itemize}
\item Input: $f\in\R[x]$, $g=(g_1,\dots,g_m)$ and $h=(h_1,\dots,h_l)$ with $g_i,h_j\in\R[x]$.
\item Output: $\hat h=(\hat h_1,\dots,\hat h_r)$ with $\hat h_t\in\R[x_{n+1}]$.
\end{itemize}
\begin{enumerate}
\item Set $\tilde h=(g_1-y_1^2,\dots,g_m-y_m^2,h)$ with $y=(y_1,\dots,y_m)$ being vector of variables independent from $x$.  
\item Compute $\hat h=(\hat h_1,\dots,\hat h_r)$ with $\hat h_t\in\R[x_{n+1}]$ such that $V(\hat h)=Z(f(V(\tilde h)))$ by using Algorithm \ref{alg:image}.
\end{enumerate}
\end{algorithm}
\begin{remark}\label{rem:complex.alg:image.semi}
By Remark \ref{rem:bound.image.variety}, the degree of each $\hat h_t$ in the output of Algorithm \ref{alg:image.semi} is bounded from above by 
\begin{equation}\label{eq:degree.bound.image.semi}
2\times d^{2^{{O((n+m)^2)}^{2^{n+m-1}}}}
\end{equation}
when $d$ is the upper bound on the degrees of $f,g_i,h_j$.
\end{remark}
\begin{lemma}\label{lem:algimag.well.semi}
Let $f\in\R[x]$, $g=(g_1,\dots,g_m)$ and $h=(h_1,\dots,h_l)$ with $g_i,h_j\in\R[x]$.
Let $\hat h=(\hat h_1,\dots,\hat h_r)$ with $\hat h_t\in\R[x_{n+1}]$ be the output of Algorithm \ref{alg:image.semi} with input $f,g,h$.
Then $V(\hat h)=Z(f(S(g)\cap V(h)))$.
\end{lemma}
\begin{proof}
The second step of Algorithm \ref{alg:image.semi} returns the correct result thanks to Lemma \ref{lem:algimag.well}.
It is sufficient to prove $f(V(\tilde h))=f(S(g)\cap V(h))$.
It follows from the following equivalences:
\begin{equation}
\begin{array}{rl}
&t\in f(V(\tilde h))\\
\Leftrightarrow & \exists (x,y)\in V(\tilde h)\,:\,t=f(x)\\
\Leftrightarrow & \exists (x,y)\in\R^{n+m}\,:\,t=f(x)\,,\,g_j(x)=y_j^2\,,\,h_i(x)=0\\
\Leftrightarrow & \exists x\in\R^{n}\,:\,t=f(x)\,,\,g_j(x)\ge 0\,,\,h_i(x)=0\\
\Leftrightarrow & \exists x\in S(g)\cap V(h)\,:\,t=f(x)\\
\Leftrightarrow&t\in f(S(g)\cap V(h))\,.
\end{array}
\end{equation}
\end{proof}
\begin{remark}\label{rem:extract.point}
If the intersection $S(g)\cap V(h)$ is finite, we can use Algorithm \ref{alg:image.semi} to find all points of this intersection as follows: 
\begin{enumerate}
\item Set $\Omega:=\emptyset$.
\item For $t=1,\dots,n$, compute $\hat h^{(t)}=(\hat h_1^{(t)},\dots,\hat h_{r_t}^{(t)})$ with $\hat h_w^{(t)}\in\R[x_{n+1}]$ such that $V(\hat h^{(t)})=Z(\pi_t(S(g)\cap V(h)))$ by using  Algorithm \ref{alg:image.semi}, where $\pi_t(x)=x_t$.
\item For every $a\in V(\hat h^{(1)})\times\dots\times V(\hat h^{(n)})$, if $g_j(a)\ge 0$, for $j=1,\dots,m$ and $h_i(a)= 0$, for $i=1,\dots,l$, then set $\Omega:=\Omega\cup \{a\}$.
\end{enumerate}
Eventually we get $\Omega=S(g)\cap V(h)$ since $S(g)\cap V(h)\subset V(\hat h^{(1)})\times\dots\times V(\hat h^{(n)})$.
Here the second step is inspired by the work of Nie and Ranestad \cite{nie2009algebraic} on algebraic degrees.
Moreover, If we replace $h$ with $h\cup\{f-f^\star\}$, where $f^\star:=\min f(S(g)\cap V(h))$, then the degrees of the entries of $\hat h^{(t)}$ give the upper bound \eqref{eq:degree.bound.image.semi} on the algebraic degree of the optimal coordinate $x_t$.
Here $d$ is the upper bound on the degrees of $f,g_i,h_j$.
\end{remark}
The following algorithm enables us to compute the dimension of the image of a given basic semi-algebraic set under a given polynomial:
\begin{algorithm}\label{alg:dim.semi}
Computing the dimension of the image of a semi-algebraic set under a polynomial:
\begin{itemize}
\item Input: $f\in\R[x]$, $g=(g_1,\dots,g_m)$ and $h=(h_1,\dots,h_l)$ with $g_i,h_j\in\R[x]$.
\item Output: $d\in\N$.
\end{itemize}
\begin{enumerate}
\item Compute $\hat h=(\hat h_1,\dots,\hat h_r)$ with $\hat h_t\in\R[x_{n+1}]$ such that $V(\hat h)=Z(f(S(g)\cap V(h)))$ by using Algorithm \ref{alg:image.semi}.  
\item Set $d:=\dim(V(\hat h)))$.
\end{enumerate}
\end{algorithm}
As a direct consequence of Lemmas \ref{lem:algimag.well.semi} and \ref{lem:dimequal}, the following lemma follows:
\begin{lemma}\label{lem:alg:dim.semi}
Let $f\in\R[x]$, $g=(g_1,\dots,g_m)$ and $h=(h_1,\dots,h_l)$ with $g_i,h_j\in\R[x]$.
Let $d$ be the output of Algorithm \ref{alg:dim.semi}  with input $f,g,h$.
Then $d$ is the dimension of the image $f(S(g)\cap V(h))$.
\end{lemma}

\section{Applications to polynomial optimization}
\label{sec:application}
\subsection{Main algorithm}
The following algorithm is our main contribution in this paper:
\begin{algorithm}\label{alg:dim.semi.opt}
Computing the optimal value of a polynomial optimization problem:
\begin{itemize}
\item Input: $f\in\R[x]$, $g=(g_1,\dots,g_m)$ and $h=(h_1,\dots,h_l)$ with $g_i,h_j\in\R[x]$.
\item Output: $\bar f^\star\in \R\cup\{\pm \infty\}$.
\end{itemize}
\begin{enumerate}
\item Compute $\hat h=(\hat h_1,\dots,\hat h_r)$ with $\hat h_t\in\R[x_{n+1}]$ such that $V(\hat h)=Z(f(S(g)\cap V(h)))$ by using Algorithm \ref{alg:image.semi}.  
\item Set $\bar f^\star:=\inf V(\hat h)$.
\end{enumerate}
\end{algorithm}
\begin{remark}\label{rem:dim.semi.opt}
By Remark \ref{rem:complex.alg:image.semi}, the degree of each $\hat h_t$ in the first step of Algorithm \ref{alg:dim.semi.opt} is bounded from above by \eqref{eq:degree.bound.image.semi}  when $d$ is the upper bound on the degrees of $f,g_i,h_j$.
Moreover, if $f,g_i,h_j\in\mathbb Q[x]$, then $\hat h_t\in \Q[x_{n+1}]$.
It is because we use a finite number of arithmetic operations on $f,g_i,h_j$ to obtain $\hat h_t$.
\end{remark}
\begin{remark}
To handle the second step of Algorithm \ref{alg:dim.semi.opt}, we need to solve the system of univariate polynomial equations $\hat h_1=\dots=\hat h_r=0$.
It might happen that $\bar f^\star=-\infty$ if $V(\hat h)=\R^n$ (e.g., $\hat h=(0)$) and $\bar f^\star=\infty$ if $V(\hat h)=\emptyset$ (e.g., $\hat h=(1)$).
We now show how to obtain exact $\bar f^\star$ in the case where $\bar f^\star\in \mathbb Q$.
Observe that $\bar f^\star$ is a common rational root of the system $\hat h_1=\dots=\hat h_r=0$.
For $t=1,\dots,r$, we use Thill's algorithm in \cite{thill2008more} to round every numerical root $a$ of the polynomial equation $\hat h_t=0$ to a rational number $\tilde a$ and then verify if the equality $\hat h_t(\tilde a)=0$ holds.
This method requires getting the numerical roots with arbitrarily high precision.
The software MPSolve in \cite{bini2020mpsolve} allows us to do so.

Interested readers can refer to \cite{li2013computing,mantzaflaris2023certified,
mantzaflaris2011continued,henrion2020real,riener2018real,
henrion2016real,henrion2015real} some newest algorithms to compute the real roots of algebraic systems.
\end{remark}

The following lemma follows directly from Lemmas \ref{lem:algimag.well.semi} and \ref{lem:dimequal}:
\begin{lemma}\label{lem:opt.well}
Let $f\in\R[x]$, $g=(g_1,\dots,g_m)$ and $h=(h_1,\dots,h_l)$ with $g_i,h_j\in\R[x]$ such that $f(S(g)\cap V(h))$ is finite.
Let $\bar f^\star$ be the output of Algorithm \ref{alg:dim.semi.opt} with input $f,g,h$.
Then $\bar f^\star=\min f(S(g)\cap V(h))$.
\end{lemma}
To prove Lemma \ref{lem:opt.well}, we remark that $f(S(g)\cap V(h))$ is finite, and hence is identical to  its Zariski closure.

\begin{remark}\label{rem:not.real.radical}
In practice, it might happen that $Z(f(S(g)\cap V(h)))\subsetneq V(\hat h)$ in the first step of Algorithm \ref{alg:dim.semi.opt} if $\check h$ in the first step of Algorithm \ref{alg:image} is not a generator of the real radical of ${I(\tilde h,x_{n+1}-f)}$, i.e., $I(\check h)\subsetneq\sqrt[\R]{I(\tilde h,x_{n+1}-f)}$, where $\tilde h$ is defined as in Step 1 of Algorithm \ref{alg:image.semi}.
If $V(\hat h)$ is finite, i.e., $V(\hat h)=\{t_1,\dots,t_r\}\subset \R$, and $S(g)\cap V(h\cup\{f-t_j\})$ is also finite for $j=1,\dots,r$, then we do the following steps to find $\bar f^\star=\min f(S(g)\cap V(h))$:
\begin{enumerate}
\item Set $A:=\emptyset$.
\item For $j=1,\dots,r$, do:
\begin{enumerate}
\item Compute $S(g)\cap V(h\cup\{f-t_j\})$ as in Remark \ref{rem:extract.point}.
\item If $S(g)\cap V(h\cup\{f-t_j\})\ne\emptyset$, set $A:=A\cup \{t_j\}$. 
\end{enumerate}
\end{enumerate}
Eventually we get $A=f(S(g)\cap V(h))$ so that $\bar f^\star=\min A$.
Moreover, $a \in S(g)\cap V(h\cup\{f-\bar f^\star\})$ is a minimizer of $\min\limits_{x\in S(g)\cap V(h)}f(x)$.
\end{remark}

In the following theorem, we apply Algorithm \ref{alg:dim.semi.opt} to find the optimal value of a polynomial optimization problem:
\begin{theorem}\label{theo:pop.app}
Let $f\in\R[x]$, $g=(g_1,\dots,g_m)$ with $g_j\in\R[x]$.
Assume that problem \eqref{eq:pop} has an optimal solution $x^\star$ and one of the following conditions holds:
\begin{enumerate}
\item The image $f(S(g)\cap C(g))$ is finite, 
\item The image $f(S(g)\cap C^+(g))$ is finite.
\item The Karush--Kuhn--Tucker conditions hold for problem \eqref{eq:pop} at $x^\star$.
\end{enumerate}
Then $f^\star$ defined as in \eqref{eq:pop} can be computed by using Algorithm \ref{alg:dim.semi.opt}.
\end{theorem}
\begin{proof}
Let  $h_{\FJ}$ be defined as in \eqref{eq:.polyFJ}.
Assume that $S(g)\cap C(g)$ is finite.
Lemma \ref{lem:image.finite} says that $f((S(g)\times\R^{m+1})\cap V(h_{\FJ}))$ is non-empty and finite.
Moreover, as shown in \cite[Lemma 1]{mai2022exact}, $(S(g)\times\R^{m+1})\cap V(h_{\FJ})$  contains $(x^\star,\bar\lambda^\star)$ for some $\bar\lambda^\star\in\R^{m+1}$.
By using Algorithm \ref{alg:dim.semi.opt} with input $f,g,h_{\FJ}$, we obtain
\begin{equation}
f^\star=\min f((S(g)\times\R^{m+1})\cap V(h_{\FJ}))=\min f(S(g))
\end{equation}
thanks to Lemma \ref{lem:opt.well}.
Here we consider $f\in\R[x,\bar\lambda]$.
Note that Algorithm \ref{alg:dim.semi.opt} produces a vector of univariate polynomials $\hat h=(\hat h_1,\dots,\hat h_r)$ such that $f^\star=\min V(\hat h)$.
By Remark \ref{rem:dim.semi.opt}, the degree of each $\hat h_t$ is bounded from above by \eqref{eq:bound.degree.univar.poly}.
Similar arguments apply to the cases where $f(S(g)\cap C^+(g))$ is finite and the Karush--Kuhn--Tucker conditions hold for problem \eqref{eq:pop} at $x^\star$.
\end{proof}

\begin{remark}
The result of Theorem \ref{theo:pop.app} requires the attainability of the optimal value $f^\star$. 
In \cite{mai2022semi}, the author proves that every polynomial optimization problem of the form \eqref{eq:pop} with finite optimal value $f^\star$ can be symbolically transformed to an equivalent problem in one-dimensional space with attained optimal value $f^\star$.
To do this, he uses quantifier elimination and algebraic algorithms that rely on the fundamental theorem of algebra and the greatest common divisor.
Let $d$ be the upper bound on the degrees of $f,g_j$. 
His symbolic algorithm has complexity $O(d^{O(n)})$ to produce the objective and constraint polynomials of degree at most $d^{O(n)}$ for the equivalent problem.
\end{remark}

%Use Hessian

\begin{lemma}\label{lem:imag}
Let $f\in\R[x]$, $g=(g_1,\dots,g_m)$ with $g_j\in\R[x]$.
Set
\begin{equation}\label{eq:hsing}
h_\text{sing}=(\sum_{j=1}^m \lambda_j \nabla g_j, \lambda_1 g_1,\dots,\lambda_m g_m,\sum_{j=1}^m\lambda_j^2-1)\,,
\end{equation}
\begin{equation}\label{eq:hsing.plus}
h_\text{sing}^+=(\sum_{j=1}^m \lambda_j^2 \nabla g_j, \lambda_1^2 g_1,\dots,\lambda_m^2 g_m,\sum_{j=1}^m\lambda_j^2-1)\,.
\end{equation}
Let $\pi:\R^{n+m}\to \R^n$ be the projection defined by 
\begin{equation}
    \pi(x,\lambda)=x\,,\,\forall x\in\R^n\,,\,\forall \lambda\in\R^{m}\,.
\end{equation}
Then the following statements hold:
\begin{enumerate}
\item $C(g)=\pi(V(h_\text{sing}))$ and $C^+(g)=\pi(V(h_\text{sing}^+))$.
\item $f(C(g))=f(V(h_\text{sing}))$ and $f( C^+(g))=f(V(h_\text{sing}^+))$.
\item $S(g)\cap C(g)=\pi((S(g)\times\R^m)\cap V(h_\text{sing}))$ and\\ $S(g)\cap C^+(g)=\pi((S(g)\times\R^m)\cap V(h_\text{sing}^+))$.
\item $f(S(g)\cap C(g))=f((S(g)\times\R^m)\cap V(h_\text{sing}))$ and\\ $f(S(g)\cap C^+(g))=f((S(g)\times\R^m)\cap V(h_\text{sing}^+))$.
\end{enumerate}
\end{lemma}
\begin{proof}
The equality $C(g)=\pi(V(h_\text{sing}))$ holds thanks to the following equivalences:
\begin{equation}
    \begin{array}{rl}
         &  x\in C(g)\\
        \Leftrightarrow & \rank(\varphi^g(x)) < m\\
        \Leftrightarrow & \exists \lambda\in\R^{m}\,:\, \sum_{j=1}^m\lambda_j^2=1\,,\, \sum_{j=1}^m \lambda_j \nabla g_j(x)=0\,,\, \lambda_j g_j =0\\
        \Leftrightarrow & \exists \lambda\in\R^{m}\,:\,(x,\lambda)\in V(h_\text{sing})\\
        \Leftrightarrow& x\in\pi(V(h_\text{sing}))\,.
    \end{array}
\end{equation}
Note that $C(g)=\pi(V(h_\text{sing}))$ implies $f(C(g))=f(V(h_\text{sing}))$ and $S(g)\cap C(g)=\pi((S(g)\times\R^m)\cap V(h_\text{sing}))$, which yields $f(S(g)\cap C(g))=f((S(g)\times\R^m)\cap V(h_\text{sing}))$.
The remaining cases are handled similarly.
\end{proof}

Next, we test the finiteness of the image of the  singularities of a given basic semi-algebraic set under a given polynomial in the following theorem:
\begin{theorem}\label{theo:dim.critical}
Let $f\in\R[x]$, $g=(g_1,\dots,g_m)$ with $g_j\in\R[x]$.
Let $h_\text{sing}$ be as in \eqref{eq:hsing} and $h_\text{sing}^+$ be as in \eqref{eq:hsing.plus}.
Then the following statements hold:
\begin{enumerate}
\item The dimensions of $f(C(g))=f(V(h_\text{sing}))$ and $f( C^+(g))=f(V(h_\text{sing}^+))$ can be computed by using Algorithm \ref{alg:dim}.
\item The dimensions of $f(S(g)\cap C(g))=f((S(g)\times\R^m)\cap V(h_\text{sing}))$ and $f(S(g)\cap C^+(g))=f((S(g)\times\R^m)\cap V(h_\text{sing}^+))$ can be computed by using Algorithm \ref{alg:dim.semi}.  
\end{enumerate}
\end{theorem}
In order to prove Theorem \ref{theo:dim.critical}, we apply directly Lemmas \ref{lem:imag}, \ref{lem:alg:dim} and \ref{lem:alg:dim.semi}.
\subsection{Examples}
\label{sec:examples}
This section provides several examples to illustrate our method in this paper.
We perform some experiments in Julia 1.7.1 with the software Oscar \cite{OSCAR}.
The codes are available in the link: 
\begin{center}
\url{https://github.com/maihoanganh/FritzJohnConds}.
\end{center}
We use a desktop computer with an Intel(R) Core(TM) i7-8665U CPU @ 1.9GHz $\times$ 8 and 31.2 GB of RAM.

\begin{example}\label{exam:opt}
Consider problem \eqref{eq:pop} with $n=m=2$, $f=x_2$ and $g=(-x_1,x_1-x_2^2)$.
Observe that the optimal value $f^\star=0$ is attained by $f$ on $S(g)$ at the unique global minimizer $(0,0)$ for problem \eqref{eq:pop}.
Moreover, the Karush--Kuhn--Tucker conditions do not hold at this minimizer.
Thus Lasserre's hierarchy \cite{lasserre2001global}  for this example does not have finite convergence  (see \cite[Proposition 3.4]{nie2014optimality}).
Let $h_{\FJ}$ be as in \eqref{eq:.polyFJ}.
Then we obtain
\begin{equation}
h_{\FJ}=(\lambda_0\begin{bmatrix}
0\\
1
\end{bmatrix}-\lambda_1\begin{bmatrix}
-1\\
0
\end{bmatrix}-\lambda_2\begin{bmatrix}
1\\
-2x_2
\end{bmatrix},-\lambda_1 x_1,\lambda_2 (x_1-x_2^2),1-\lambda_0^2-\lambda_1^2-\lambda_2^2)\,.
\end{equation}
We now use Algorithm \ref{alg:dim.semi.opt} to find $f^\star$ as follows:
\begin{itemize}
\item Step 1 of Algorithm \ref{alg:dim.semi.opt}: Compute $\hat h=(\hat h_1,\dots,\hat h_r)$ with $\hat h_t\in\R[x_{n+1}]$ such that $V(\hat h)=Z(f((S(g)\times\R^{m+1})\cap V(h_{\FJ})))$ by using Algorithm \ref{alg:image.semi}.  
\begin{itemize}
\item Step 1 of Algorithm \ref{alg:image.semi}: Set $\tilde h=(g_1-y_1^2,g_2-y_2^2,h_{\FJ})$ with $y=(y_1,y_2)$ being vector of variables independent from $x$.  Then we get
\begin{equation}
\tilde h=\begin{bmatrix}
-x_{1} - y_{1}^2\\
 x_{1} - x_{2}^2 - y_{2}^2\\
 \lambda_1 - \lambda_{2}\\
 2x_{2}\lambda_2 + \lambda_0\\
 -x_{1}\lambda_1\\
 x_{1}\lambda_2 - x_{2}^2\lambda_2\\
 -\lambda_0^2 - \lambda_1^2 - \lambda_2^2 + 1\\
 -x_{2} + x_{3}\\
\end{bmatrix}\,.
\end{equation}
\item Step 2 of Algorithm \ref{alg:image.semi}: Compute $\hat h=(\hat h_1,\dots,\hat h_r)$ with $\hat h_t\in\R[x_{n+1}]$ such that $V(\hat h)=Z(f(V(\tilde h)))$ by using Algorithm \ref{alg:image}.
\begin{itemize}
\item Step 1 of Algorithm \ref{alg:image}: Compute a generator $\check h\subset \R[x]$ of the real radical $\sqrt[\R]{I(\tilde h,x_3-f)}$. 
Then we obtain\footnote{We use the subroutine {\tt radical} of Oscar to compute radical generators of complex varieties although it might happen that the outputs of this subroutine are not real radical generators.}
\begin{equation}
\check h=(
\lambda_1 - \lambda_2,
 \lambda_0,
 y_{2},
 y_{1},
 x_{3},
 x_{2},
 x_{1},
 2\lambda_2^2 - 1)\,.
\end{equation} 
\item Step 2 of Algorithm \ref{alg:image}: Compute a Gr\"obner basis $h^G$ of $I(\check h)$ w.r.t. lex order where
\begin{equation}
x_1\succ x_2\succ y_1\succ y_2\succ \lambda_0\succ \lambda_1\succ \lambda_2\succ x_3\,.
\end{equation}
Then we get $h^G=\check h$.
\item Step 3 of Algorithm \ref{alg:image}: Set $\hat h:=h^G\cap \R[x_{3}]$.
Then we have $\hat h=\{x_3\}$.
\end{itemize}
\end{itemize}
\item Step 2 of Algorithm \ref{alg:dim.semi.opt}: Set $\bar f^\star:=\min V(\hat h)$. Thus it implies that $\bar f^\star=0=f^\star$.
\end{itemize}
By \eqref{eq:equalities.varieties} (with $h=\tilde h$), if $(\bar x,\bar y,\bar \lambda)\in V(h^G)$, then $\bar x$ is a minimizer of problem \eqref{eq:pop}.
In this case, we obtain $\bar x=(0,0)$.
\end{example}

\if{
\begin{remark}
In Example \ref{exam:opt}, if we set $\check h=(\tilde h, x_3-f)$ in Step 1 of Algorithm \ref{alg:image}, then we obtain
\begin{equation}
h^G=
\begin{bmatrix}
 \lambda_1 - \lambda_2\\
 x_{2} - x_{3}\\
 x_{1}\\
 2\lambda_2^2 - 1\\
 x_{3} + \lambda_0\lambda_2\\
 2x_{3}\lambda_2 + \lambda_0\\
 \lambda_0^2 + 2\lambda_2^2 - 1\\
 x_{3}\lambda_0\\
 y_{2}^2\\
 x_{1} + y_{1}^2\\
 -x_{1} + x_{3}^2 + y_{2}^2
\end{bmatrix}
\,.
\end{equation}
It implies that $\hat h=h^G\cap \R[x_{3}]=\emptyset$.
\end{remark}
}\fi

Next, we show in the following example how to compute the dimension of the image of the singularities of a semi-algebraic set under a given polynomial:
\begin{example}\label{exam:compute.dim}
Consider problem \eqref{eq:pop} with $n=2$, $m=3$, $f=x_1-5x_2$ and $g=(x_1^2-x_2,-x_1^2+4x_2,-x_2+1)$.
Observe that the optimal value $f^\star=-7$ is attained by $f$ on $S(g)$ at the unique global minimizer $(-2,1)$ for problem \eqref{eq:pop}.
It is not hard to check that the Karush--Kuhn--Tucker conditions hold for problem \eqref{eq:pop} at this minimizer.
(Moreover, the Fritz John conditions hold for problem \eqref{eq:pop} at a local minimizer $(0,0)$ but the Karush--Kuhn--Tucker conditions do not hold for problem \eqref{eq:pop} at this point.)
Let $h_\text{sing}$ be as in \eqref{eq:hsing}.
Then we obtain
\begin{equation}
\begin{array}{rl}
h_\text{sing}=&(\lambda_1\begin{bmatrix}
2x_1\\
-1
\end{bmatrix}+\lambda_2\begin{bmatrix}
-2x_1\\
4
\end{bmatrix}+\lambda_3\begin{bmatrix}
0\\
-1
\end{bmatrix},\\\\
&\lambda_1(x_1^2-x_2),\lambda_2 (-x_1^2+4x_2),\lambda_3(-x_2+1),1-\lambda_0^2-\lambda_1^2-\lambda_2^2-\lambda_3^2)\,.
\end{array}
\end{equation}
We now use Algorithm \ref{alg:dim.semi} to find $\dim(f(C(g))\cap S(g))$ as follows:
\begin{itemize}
\item Step 1 of Algorithm \ref{alg:dim.semi}: Compute $\hat h=(\hat h_1,\dots,\hat h_r)$ with $\hat h_t\in\R[x_{n+1}]$, such that $V(\hat h)=Z(f((S(g)\times \R^m)\cap V(h_\text{sing})))$ by using Algorithm \ref{alg:image.semi}. 
 \begin{itemize}
\item Step 1 of Algorithm \ref{alg:image.semi}: Set $\tilde h=(g_1-y_1^2,g_2-y_2^2,g_3-y_3^2,h_\text{sing})$ with $y=(y_1,y_2,y_3)$ being vector of variables independent from $x$.  Then we get
\begin{equation}
\tilde h=\begin{bmatrix}
x_{1}^2 - x_{2} - y_{1}^2\\
 -x_{1}^2 + 4x_{2} - y_{2}^2\\
 -x_{2} - y_{3}^2 + 1\\
 2x_{1}\lambda_{1} - 2x_{1}\lambda_{2}\\
 -\lambda_{1} + 4\lambda_{2} - \lambda_{3}\\
 x_{1}^2\lambda_{1} - x_{2}\lambda_{1}\\
 -x_{1}^2\lambda_{2} + 4x_{2}\lambda_{2}\\
 -x_{2}\lambda_{3} + \lambda_{3}\\
 -\lambda_{1}^2 - \lambda_{2}^2 - \lambda_{3}^2 + 1\\
 -x_{1} + 5x_{2} + x_{3}
\end{bmatrix}\,.
\end{equation}
\item Step 2 of Algorithm \ref{alg:image.semi}: Compute $\hat h=(\hat h_1,\dots,\hat h_r)$ with $\hat h_t\in\R[x_{n+1}]$ such that $V(\hat h)=Z(f(V(\tilde h)))$ by using Algorithm \ref{alg:image}.
\begin{itemize}
\item Step 1 of Algorithm \ref{alg:image}: Compute the generator $\check h\subset \R[x]$ of the real radical $\sqrt[\R]{I(\tilde h,x_3-f)}$. 
Then we obtain
\begin{equation}
\check h=(
\lambda_{3},
 \lambda_{1} - 4\lambda_{2},
 y_{2},
 y_{1},
 x_{3},
 x_{2},
 x_{1},
 17\lambda_{2}^2 - 1,
 y_{3}^2 - 1)\,.
\end{equation} 
\item Step 2 of Algorithm \ref{alg:image}: Compute a Gr\"obner basis $h^G$ of $I(\check h)$ w.r.t. lex order where
\begin{equation}
x_1\succ x_2\succ y_1\succ y_2\succ y_3\succ \lambda_1\succ \lambda_2\succ x_3\,.
\end{equation}
Then we get $h^G=\check h$.
\item Step 3 of Algorithm \ref{alg:image}: Set $\hat h:=h^G\cap \R[x_{3}]$.
Then we have $\hat h=\{x_3\}$.
\end{itemize}
\end{itemize}
\item Step 2 of Algorithm \ref{alg:dim.semi}: Set $d:=\dim(V(\hat h)))$.
Then we get $d=\dim(\{0\})=0$.
\end{itemize}
\end{example}

\begin{example}\label{exam:MPCC}
Consider the following mathematical program with complementarity constraints given in \cite[(1.2)]{guo2021mathematical}:
\begin{equation}\label{eq:MPCC}
\begin{array}{rl}
\min & \sum_{j=1}^n z_j^{p}\\
\text{s.t.} & z\ge 0\,,\,Mz+q\ge 0\,,\,z^\top (Mz+q)=0\,,
\end{array}
\end{equation}
where $p\in (0,1)$, $z=(z_1,\dots,z_n)$, $M\in\R^{n\times n}$, and $q\in\R^n$.
Let $p=\frac{1}{2}$ and set $x_j=z_j^p$.
Then $z=x^2=(x_1^2,\dots,x_n^2)$ and problem \eqref{eq:MPCC} can be written as the form \eqref{eq:pop} with $f=\sum_{j=1}^n x_j$ and 
$g=(x,Mx^2+q,-x^{2\top} (Mx^2+q))$.
We now take $n=2$, $M=\begin{bmatrix}
-2  &-1\\
 -1 & -4
\end{bmatrix}$, $q=-M\times \begin{bmatrix}
1\\
1
\end{bmatrix}=\begin{bmatrix}
3\\
5
\end{bmatrix}$ in this example.
It is not hard to prove that problem \eqref{eq:pop} has optimal value $f^\star=0$ and a global minimizer $x^\star=(0,0)$.
By using Algorithm \ref{alg:dim.semi.opt}, we obtain this exact output.
\end{example}

\paragraph{Additional examples.} We display additional results of Algorithm \ref{alg:dim.semi.opt} in Table \ref{tab:bench}.
%We indicate the data of problem \eqref{exam:opt}. 
Here ``$\bar f^\star$'' and ``time''
correspond to the value returned by our method and the running time in seconds to obtain this value, respectively.
\begin{table}
\footnotesize
\caption{\footnotesize Computing $f^\star=\min f(S(g))$.}
\label{tab:bench}
\begin{center}
\begin{tabular}{ c|cccccccc }
input & $n$ & $m$ & $\deg(f)$ &$(\deg(g_j))_{j=1}^m$& $h$ & $f^\star$& $\bar f^\star$ & time \\
\hline
Example \ref{exam:opt} & 2 &2 & 1& (1,2) & $h_{\FJ}$& 0&0& 0.01 \\
Example \ref{exam:compute.dim} & 2 &3 & 1& (2,2,1) & $h_{\KKT}$& -7&-7& 0.01 \\
Example \ref{exam:MPCC} & 2 &5 & 1& (1,1,2,2,4) & $h_{\FJ}$& 0&0& 96 \\
\cite[Example 3]{mai2022exact} & 2 &1 & 2& (3) & $h_{\FJ}$& 0&0& 0.5 \\
\cite[Example 5]{mai2022exact} & 1 &1 & 1& (2) & $h_{\FJ}$& 0&0& 0.02 \\
\cite[Example A.2]{greuet2014probabilistic} & 2& 1&4 &(3) &$h_{\KKT}$& -1 & -1 & 0.04\\
\cite[Example 11]{mai2022exact} & 3 &1 & 6& (2) & $h_{\KKT}$& 0&0& 0.1 \\
\cite[Example 12]{mai2022exact} & 2 &3 & 3& (1,1,1) & $h_{\KKT}$& 0&0& 0.01 \\
\cite[Example 13]{mai2022exact} & 3 &3 & 6& (2,2,2) & $h_{\KKT}^+$& 0&0& 1 \\
\cite[Example 18]{mai2022exact} & 2 &3 & 1& (1,1,2) & $h_{\FJ}^+$& 0&$-\infty$& 0.1 \\
\cite[Example 9]{mai2022exact} & 2 &1 & 1& (3) & $h_{\FJ}^+$& 0&$+\infty$& 0.02 \\
\end{tabular}
\end{center}
\end{table}

The Karush--Kuhn--Tucker conditions do not hold for Example \ref{exam:opt} and \cite[Examples 3 and 5]{mai2022exact} at their minimizers.
In \cite[Examples 11, 12, and 13]{mai2022exact}, Lasserre's hierarchy \cite{lasserre2001global} does not have finite convergence.
Note that \cite[Example 5]{mai2022exact} has infinite number of singularities so that the regularity assumption for the method of Greuet and Safey El Din \cite{greuet2014probabilistic} does not hold.
Table \ref{tab:bench} shows that our value $\bar f^\star$ is the same as the exact optimal value $f^\star$ except the last two examples.
Here \cite[Example 3]{mai2022exact} belongs to the case mentioned in Remark \ref{rem:not.real.radical}.
In \cite[Section 4.2]{mai2022exact}, we only obtain approximate values with low accuracy for this example when using an interior-point method.
Note that \cite[Example 18]{mai2022exact} does not satisfy the finiteness assumption of $f(S(g)\cap C^+(g))$ and $\hat h$ in Step 1 of Algorithm \ref{alg:dim.semi.opt} is empty, which implies $\bar f^\star=\inf V(\hat h)=\inf \R=-\infty$.
In \cite[Example 9]{mai2022exact}, problem \eqref{eq:pop} does not have any minimizer and $\hat h$ in Step 1 of Algorithm \ref{alg:dim.semi.opt} is the singleton $\{1\}$, which implies $\bar f^\star=\inf V(\hat h)=\inf \emptyset=+\infty$.

\paragraph{Acknowledgements.}
The author was supported by the MESRI funding from EDMITT.
%\bibliography{references} 
\bibliographystyle{abbrv}

\end{document}